\documentclass[12pt, a4paper]{article}

\usepackage[utf8]{inputenc}
\usepackage[T1]{fontenc}
\usepackage[english]{babel}

\usepackage[left=3cm, right=3cm, top=4cm]{geometry}

\usepackage[dvipsnames]{xcolor}
\usepackage{pgf,tikz}
\usetikzlibrary{arrows}

\usepackage[format=hang,font=small]{caption}
\usepackage{booktabs,bm,multirow,subfig}%
\usepackage{enumitem,comment}
\usepackage{mathrsfs} 

\usepackage{amsmath,amsfonts,amssymb,amsthm}%
\usepackage{dsfont,multicol}
\usepackage{pifont}

\renewcommand{\phi}{\varphi}
\renewcommand{\theta}{\vartheta}
\renewcommand{\epsilon}{\varepsilon}

\newcommand{\R}{\mathbb{R}}
\newcommand{\Z}{\mathbb{Z}}
\newcommand{\N}{\mathbb{N}}
\newcommand{\I}{\mathbb{I}}
\newcommand{\Null}{\mathcal{N}}
\newcommand{\Es}{\mathscr{E}}
\newcommand{\Hs}{\mathscr{H}}
\DeclareMathOperator{\rot}{rot}

\providecommand{\abs}[1]{\left\lvert#1\right\rvert} 
\newcommand{\clos}[1]{\overline{#1}}

\newcommand*{\dd}{\mathop{}\!\mathrm{d}}

\newtheorem{theorem}{Theorem}
\newtheorem{lemma}[theorem]{Lemma}
\newtheorem{prop}[theorem]{Proposition}
\newtheorem{corol}[theorem]{Corollary}
\theoremstyle{definition}

\newtheorem{remark}[]{Remark}

\numberwithin{equation}{section}

\usepackage{fancyhdr}
\title{Existence of a periodic solution for superlinear second order ODEs}

\author{\textsc{Paolo Gidoni}\\
\small
	\textit{Institute of Information Theory and Automation, Czech Academy of Sciences},\\\small
\textit{Pod vodárenskou veží 4, CZ-182 08 Prague, Czech Republic,} \texttt{gidoni@utia.cas.cz}
}
\date{}

\begin{document}

	\maketitle
	
\begin{abstract}
We prove a necessary and sufficient condition for the existence of a $T$-periodic solution for the time-periodic second order differential equation $\ddot{x}+f(t,x)+p(t,x,\dot x)=0$, where $f$ grows superlinearly in $x$ uniformly in time, while $p$ is bounded. Our method is based on a fixed-point theorem which uses the rotational properties of the dynamics. 
\end{abstract}	{\small
}

\section{Introduction}

In this paper we discuss the existence of a periodic solution for the second order superlinear equation
\begin{equation} \label{eq:introODE}
	\ddot{x}+f(t,x)+p(t,x,\dot x)=0
\end{equation}
where $f$ and $p$ are Carathéodory functions $T$-periodic in time, and the uniqueness of solution for the associated Cauchy problems is assumed. The function $f$ satisfies the superlinear growth condition
\begin{equation*}
	\lim_{\abs{x}\to+\infty}\frac{f(t,x)}{x}=+\infty \qquad\text{uniformly in $t\in[0,T]$\,,}
\end{equation*}
whereas $\abs{p(t,x,y)}<\gamma_p(t)+C_p\abs{x}$ for some positive constant $C_p$ and  integrable function $\gamma_p$.

In the case of an autonomous superlinear term $f=f(x)$, the existence of a $T$-periodic solution for \eqref{eq:introODE} has been extensively studied. The first contributions to the topic are due to Morris (e.g., \cite{Mor1,Mor2}) and Ehrmann \cite{Ehr}. The first focused, in a series of papers, on Duffing's equation $f=x^3$ and $p=p(t)$, while the second studied the problem for $p=p(t)$ under some additional symmetry conditions. After various contribution to the problem by several authors (cf.~the survey in \cite{CaMaZa}), a key step forward is due to 
Fučík and Lovicar \cite{Fucik}, who proved the existence of a periodic solution in the general case $f=f(x)$, $p=p(t)$ . Their strategy was later improved by Struwe \cite{Stru}, allowing a generic term $p=p(t,x,\dot x)$. We also mention \cite{CaMaZa} for an alternative proof of this last result based on a continuation method, with generalizations to planar and higher dimensional systems.

If the superlinear term $f(t,x)$ is nonautonomous, the situation is much more problematic and  unexplored. The key issue is the global existence of solutions, which provides the well-posedness of the Poincaré time-map for any initial data, thus facilitating a fixed-point approach.

In the autonomous case $f=f(x)$, global existence of solutions can be obtained via an a priori  estimate of the energy $\dot x^2/2+V(x)$, where $V$ is a primitive of $f$.
Such strategy cannot be applied, in general, to the nonautonomous case $f(t,x)$, where global existence may possibly not hold true. An example for which a solution of \eqref{eq:introODE}  explodes to infinity in finite time is presented in \cite{CofUll}. 
The aim of this paper is to deal with this issue, developing a topological approach compatible with the possible non-continuability in $[0,T]$ of some solutions of \eqref{eq:introODE}.

Our main result (Theorem \ref{th:chall}) is the existence of a $T$-periodic solution for \eqref{eq:introODE}, under the only additional assumption \ref{hyp:Nexist}, requiring that for every $\bar t\in[0,T]$  the solution of the Cauchy problem $x(\bar t)=\dot x(\bar t)=0$ can be continued to $[0,T]$.  We notice that \ref{hyp:Nexist} does not imply global existence of solutions: for instance, it is satisfied also by the counterexample to global existence in \cite{CofUll}, since this admits the constant solution $x=0$. Indeed, as pointed out in Remark~\ref{rem:a5calc}, \ref{hyp:Nexist} can be verified, for instance, by a growth condition on $f$ valid only in a suitable bounded region.

In Corollary \ref{cor:chall_star}, we show that \ref{hyp:Nexist} can be weakened to a \emph{necessary and sufficient} condition \ref{hyp:Nexist_star} for the existence of a $T$-periodic solution. More precisely, the Cauchy condition for which we require continuability can be generalized to $x(\bar t)=a(\bar t), \dot x(\bar t)=\dot a(\bar t)$ for some smooth, $T$-periodic function $a$. As we will discuss in Remark~\ref{rem:a5idea}, such condition has a clear topological meaning in light of the desirable rotational properties of the dynamics.
We are not aware  whether a counterexample to \ref{hyp:Nexist_star} is possible for \eqref{eq:introODE}, although, if it exists, we expect it to be rare and pathological, cf.~Remark~\ref{rem:a5conj}.

In the Hamiltonian case $p\equiv 0$, the existence of a (first) periodic solution implies the existence of infinitely many additional $T$-periodic solutions, as we show in Theorem \ref{th:ham}. This fact was first noticed by Hartman \cite{Hart} and Jacobowitz \cite{Jacob}.  They considered the case $f(t,0)=0$, hence assuming $x\equiv 0$ as a first periodic solution, and obtained infinitely many others, rotating around the first one, by iterate applications of the Poincaré--Birkhoff Theorem.  We refer to \cite{DingZan,Bos12} for generalizations on the plane and to \cite{BosOrt,FonSfe,FonGid} for the extension to systems of differential equations.
Outside the Hamiltonian case, such results are in general no longer true, since the constant null solution may be the unique $T$-periodic solution, cf.~e.g.~\cite[Eq.~(1.13)]{Hart}.

To conclude, we observe that the existence of a periodic solution is entangled with the more general issue concerning boundedness, or lack thereof, of solutions for \eqref{eq:introODE}.  The problem was first considered by Littlewood  \cite{Lit} for Duffing's equation with a general bounded forcing $p(t)$, and has inspired an ample literature. For a periodic forcing, the boundedness of all solutions have been proved for polynomial generalizations of Duffing's equation \cite{DieZeh}.
However, also for this issue, the situation is more complex and unexplored in the case of a general non-autonomous $f=f(t,x)$.   Clearly, boundedness of all solutions cannot be expected, as shown by the non-continuability example in \cite{CofUll} previously mentioned. On the other hand, in \cite{Ver} the existence of infinite bounded solutions has been proved  for $p\equiv 0$ and without the periodicity assumption in $t$. Since periodic solutions are obviously bounded, our main result could be seen also as a tiny contribution in this direction.


\section{Notation and preliminary results} \label{sec:prel}

We denote with $\clos U$ the closure of a set $U$ and with $\partial U$ its boundary. Open intervals are denoted with $(a,b)$, closed ones with $[a,b]$, and the mixed case $(a,b]$ is defined accordingly.

Most of our analysis will be performed on the plane $\R^2$. We denote with $\I_{\R^2}$ the identity map in $\R^2$ and with $0_{\R^2}$ the origin $(0,0)$. We introduce clockwise polar coordinates $(\theta, r)$, according to the change of coordinates  $\Phi\colon\clos\Hs\to \R^2$
\begin{equation}\label{eq:Fpolar}
	(x,y)=\Psi(\theta,r)=(r\cos\theta, -r \sin \theta)\,,
\end{equation}
where $\Hs=\R\times (0,+\infty)$ and $\clos\Hs=\R\times [0,+\infty)$ denote the open and closed halfplanes, respectively.
We denote with $B_R\subset \R^2$ the open ball with radius $R>0$ centred in the origin $0_{\R^2}$, namely $B_R=\Psi(\R\times[0,R) )$. For any vector $v\in\R^n$, $\abs v$ denotes its Euclidean norm.

Let us consider  two  sets $E\subseteq D\subseteq \R^2$  and the evolution described by a continuous function $\phi\colon [0,T]\times D\to \R^2$, such that $\phi(0,z)=z$ for every $z\in D$. Let us suppose that
\begin{equation}\label{eq:lift_cond}
	\phi(t,z)\neq 0_{\R^2} \qquad\text{for every $(t,x)\in[0,T]\times E$}
\end{equation}
namely that each $\phi$-orbit of the points in $E$ does not cross the origin of $\R^2$.
For convenience's sake, we introduce the set $\Null\subset \R^2$ as
\begin{equation}\label{eq:defNull}
	\Null:=\{x\in D : \phi(t,x)=0 \quad\text{for some $t\in[0,T]$}\}\,,
\end{equation}
so that the assumption \eqref{eq:lift_cond} can be written as $E \cap \Null=\emptyset$. 

Let us write $\Es=\Psi^{-1}(E)\subset \Hs$. Since $(\Psi,\Hs)$ is a covering space of $\R^2\setminus\{0\}$,  by the \emph{homotopy lifting property} (cf. e.g. \cite[Chapter 1.3]{Hat}) we define $\Phi\colon [0,T]\times\Es\to\Hs$ as the unique continuous lift in clockwise polar coordinates  of the map $\phi$. Namely, $\Phi=\Phi(t,(\theta,r))$ is the unique continuous map such that for every $\eta\in \Es$ and $t\in[0,T]$ holds:
\begin{equation}
\Phi(0,\eta)=\eta \qquad \text{and} \qquad	\Psi(\Phi(t,\eta))=\phi(t,\Psi(\eta))\,.
\end{equation}
As a consequence of the periodicity of $\Psi$ in $\theta$, for every $(t,(\theta,r))\in[0,T]\times\Es$ we have
\begin{equation}\label{eq:translinvar}
	\Phi(t,(\theta+2\pi,r))=\Phi(t,(\theta,r))+(2\pi,0)\,.
\end{equation}
We denote the two components of $\Phi$ as $\Phi=(\Phi^\theta, \Phi^r)$.

Always assuming \eqref{eq:lift_cond}, we define the rotation associated to the point $z\in E$ during the interval $[0,T]$ as
\begin{equation}\label{eq:rotdef}
	\rot_\phi (z) := \frac{\Phi^\theta(T,\eta_z)-\Phi^\theta(0,\eta_z)}{2\pi}
\end{equation}
for any $\eta_z\in \Psi^{-1}(z)$. The definition is well-posed by \eqref{eq:translinvar}.
Accordingly, we write $\rot_\phi (E) :=\{\rot_\phi(z) : z\in E \}$. 

Our main result is based on the following fixed-point theorem.
\begin{theorem} \label{th:deg1}
	Let $U\subset\R^2$ be an open bounded set with $0_{\R^2}\in U$. Assume that  $\phi\colon[0,T] \times\clos U\to\R^2$ is a continuous map  such that $\phi(0,z)=z$ for every $z\in \clos U$ and $\Null\cap \partial U=\emptyset$. If
	\begin{equation}\label{eq:deg1rot}
	\rot_\phi(\partial U)\cap \Z=\emptyset
	\end{equation}
	then $\deg(\phi(T,\cdot)-\I_{\R^2},U,0)=1$, and in particular $\phi(T,\cdot)$ has a fixed point in~$U$.
\end{theorem}

Theorem~\ref{th:deg1} can be obtained, as we show below, as a corollary of the Poincaré--Bohl fixed-point Theorem, cf.~for instance \cite{FonGid0}. We refer to \cite{Gid21} for an alternative proof of Theorem~\ref{th:deg1} within a more general framework connecting rotational properties and Brouwer's degree.

\begin{theorem}[Poincaré--Bohl]
	Let $U\subset\R^d$ be an open bounded set with $0_{\R^d}\in U$. Assume that $P\colon \clos U\to \R^2$ is a continuous map such that
	\begin{equation}\label{eq:PBohl}
		P(z)\neq \lambda z \qquad \text{for every $z\in \partial U$ and every $\lambda\in[1,+\infty)$}\,.
	\end{equation}
Then $\deg(P-\I_{\R^d},U,0)=(-1)^d$, and in particular $P$ has a fixed point in~$U$.
\end{theorem}

\begin{proof}[Proof of Theorem \ref{th:deg1}]
	Under the assumptions of Theorem \ref{th:deg1}, we notice that $\phi(T,z)=\lambda z$ for some $\lambda>0$ if and only if $\rot_\phi(z)\in\Z$. Hence \eqref{eq:deg1rot} implies \eqref{eq:PBohl} and therefore Theorem \ref{th:deg1} follows straightforwardly by the Poincaré--Bohl Theorem for $d=2$.
\end{proof}

To conclude this section, we state a result on continuous dependence of solutions for ordinary differential equations, cf.~for instance \cite[Ch.~1, $\mathsection$1, Th. 6]{Filippov}. As usual, the dot $\dot{}$ denotes the derivative with respect to the time variable $t$.

First, we recall that a function $F=F(t,x)\colon \R\times \R^d$ satisfies the \emph{Carathéodory property} if the following conditions hold:
\begin{itemize}
	\item for every $x\in\R^d$ the function $F(\cdot,x)$ is measurable in $t$;
	\item for almost every $t\in\R$ the function $F(t,\cdot)$ is continuous in $x$;
	\item for every compact set $D\subset \R\times \R^d$ there exists a Lebesgue integrable function $m_D(t)$ such that $\abs{F(t,x)}\leq m_D(t)$ for every $(t,x)\in D$.
\end{itemize}
Given a function $F(t,z)$ satisfying the Carathéodory property, solutions of the differential equation $\dot z=F(t,z)$ are intended in the usual weaker sense \cite{Filippov,Hale}.

\begin{theorem}[Continuous dependence]\label{th:contdep}
	Let $h=h(t,z)\colon\R\times\R^k\to\R^k$ satisfy the Carathéodory property. Assume that the solution  $z(t;t_0,z_0)$ of the problem
	\begin{equation*}
		\begin{cases}
			\dot z=h(t,z)\\
			z(t_0)=z_0
		\end{cases}
	\end{equation*}
	exists on the interval $[t_a,t_b]$ and is unique. Then there exists $\delta>0$ such that, for every $(\bar t,\bar z)$ such that $\abs{t_0-\bar t}+\abs{z_0-\bar z}<\delta$, the solution $z(t;\bar t,\bar z)$ of the Cauchy problem defined by $z(\bar t)=\bar z$ exists on the interval $[t_a,t_b]$. Moreover, $z(s;\bar t,\bar z)$ is a continuous function of $(s;\bar t,\bar z)$ at $(t;t_0,z_0)$.
\end{theorem}


\section{Statement of the main result} \label{sec:chall}

We study the second order ordinary differential equation
\begin{equation}\label{eq:Fode}
	\ddot{x}+f(t,x)+p(t,x,\dot x)=0
\end{equation}
under the following assumptions:
\begin{enumerate}[label=\textup{(A\arabic*)}]
	\item \label{hyp:reg} the functions $f\colon\R\times \R \to\R$ and $p\colon \R\times \R^2\to\R$ both satisfy the Carathéodory property and are $T$-periodic in the first variable $t$;
	\item \label{hyp:uniqueness} For every $(t_0,x_0,y_0)\in[0,T]\times \R^2$ there is uniqueness of solution for the Cauchy problem $x(t_0)=x_0$, $\dot x(t_0)=y_0$ associated to \eqref{eq:Fode};
	\item \label{hyp:growth} $\displaystyle\lim_{\abs{x}\to+\infty}\frac{f(t,x)}{x}=+\infty$  uniformly in $t\in[0,T]$;
	\item  \label{hyp:plim} there exist a positive constant $C_p$ and a positive function $\gamma_p\in L^1([0,T],\R)$ such that
	\begin{equation*}
		\abs{p(t,x,y)}<\gamma_p(t)+C_p\abs{x} \qquad\text{for every $(t,x,y)\in\R^3$\,;}
	\end{equation*}
	\item \label{hyp:Nexist} for every $\bar t\in[0,T]$, the solution of the Cauchy problem
	\begin{equation*}
		\begin{cases}
			\ddot{x}+f(t,x)+p(t,x,\dot x)=0\\
			x(\bar t)=\dot{x}(\bar t)=0
		\end{cases}
	\end{equation*}
	can be continued to the whole interval $[0,T]$.
\end{enumerate}

The main result of the paper is the following.
\begin{theorem}\label{th:chall}
	If assumptions \ref{hyp:reg},\ref{hyp:uniqueness},\ref{hyp:growth},\ref{hyp:plim} and \ref{hyp:Nexist} hold, then Equation \eqref{eq:Fode} admits at least one $T$-periodic solution.
\end{theorem}

Notice that, as a corollary of Theorem \ref{th:chall}, it is actually possible to weaken assumption \ref{hyp:Nexist}, obtaining a necessary and sufficient condition, as follows.

\begin{corol}\label{cor:chall_star}
	Suppose that \ref{hyp:reg},\ref{hyp:uniqueness},\ref{hyp:growth},\ref{hyp:plim} hold. Then  
	\begin{enumerate}[label=\textup{(A5$^*$)}]
		\item there exists a $T$-periodic function $a\colon \R\to \R$, with absolutely continuous derivative $\dot a(t)$, such that for every $\bar t\in[0,T]$, the Cauchy problem
		\begin{equation*}
			\begin{cases}
				\ddot{x}+f(t,x)+p(t,x,\dot x)=0\\
				x(\bar t)=a(\bar t),\quad \dot{x}(\bar t)=\dot a(\bar t)
			\end{cases}
		\end{equation*}
		admits a solution continuable to the whole interval $[0,T]$; \label{hyp:Nexist_star}
	\end{enumerate}	
	is a necessary and sufficient condition for the existence  of a $T$-periodic solution for Equation \eqref{eq:Fode}.
\end{corol}
\begin{proof}
	To see that \ref{hyp:Nexist_star} is sufficient, we  consider the time-dependent change of variable $\tilde x(t)=x(t)-a(t)$. Setting $\tilde f(t,u)=f(t,u+a(t))$ and $\tilde p(t,u,w)=p(t,u+a(t),w+\dot a(t))$ we can apply Theorem \ref{th:chall} to
	\begin{equation}\label{eq:Fode_trasl}
		\ddot{\tilde x}+\tilde f(t,\tilde x)+\tilde p(t,\tilde x,\dot{\tilde x})=0\,.
	\end{equation}
	The $T$-periodic solution $\tilde x^*$ of \eqref{eq:Fode_trasl} thus obtained corresponds to a $T$-periodic solution $x^*=\tilde x^*+a(t)$ of \eqref{eq:Fode}.
	
	On the other hand, if \eqref{eq:Fode} admits a $T$-periodic solution $\bar x(t)$, condition \ref{hyp:Nexist_star} is trivially satisfied for $a(t)=\bar x(t)$.
\end{proof}

In the unperturbed case $p\equiv 0$,  Equation \eqref{eq:Fcauchy} define an Hamiltonian structure on the phase plane, and the existence of a (first) periodic solution leads to the existence of infinitely many additional $T$-periodic solutions.  We present now such a result, assuming for simplicity a slightly more regular framework than above.

\begin{theorem}\label{th:ham}
Let  $f\colon\R\times \R \to\R$ be continuous, $T$-periodic in the first variable $t$ and continuously differentiable in the second variable $x$; moreover, set $p\equiv 0$. If \ref{hyp:uniqueness},\ref{hyp:growth} and \ref{hyp:Nexist_star} hold, then Equation \eqref{eq:Fode} has infinitely many $T$-periodic solutions.
\end{theorem}
\begin{proof}
	By Corollary~\ref{cor:chall_star}, we deduce that \eqref{eq:Fode} admits a $T$-periodic solution $\bar{x}(t)$. We now consider the differential equation 
\begin{equation}\label{eq:modFode}
	\ddot{w}+f(t,w+\bar x(t))+\ddot{\bar x}(t)=0\,.
\end{equation}
We notice that $w\equiv 0$ is a solution of \eqref{eq:modFode}. Moreover, $\hat{w}$ is a $T$-periodic solution of \eqref{eq:modFode} if and only $\hat{w}+\bar x$ is a $T$-periodic solution of \eqref{eq:Fode}. The existence of infinitely many $T$-periodic solutions for \eqref{eq:modFode} is provided by \cite[Theorem~1.1]{Hart} (in alternative, see \cite[Theorem~1]{FonSfe}), thus concluding the proof.
\end{proof}

\begin{remark} \label{rem:a5calc} We presented \ref{hyp:Nexist} in a general abstract form; we notice however that it can be obtained by mild bounds on the radial component of \eqref{eq:Fcauchy}  on a bounded region. For instance, a sufficient condition for \ref{hyp:Nexist} is the following: 	\begin{quote}\textit{there exists $\alpha>0$ such that 
			$y(x-f(t,x,y)-p(t,x,y))<\alpha \abs{(x,y)}$ for every $(t,x,y)\in [0,T]\times B_{\alpha T}$}\,.
	\end{quote}
	Other sufficient conditions can be obtained analogously, for instance, replacing the uniform bound on the radial component with a suitable growth condition, or the parameter $\alpha$ with an integrable function $\alpha(t)$. We emphasise that \ref{hyp:Nexist} is much weaker than the existence on $[0,T]$ of all the solutions of \eqref{eq:Fcauchy} (which is equivalent global existence of solution), since it involves only a bounded region. In this sense, we observe that the counterexample to global existence in \cite{CofUll} trivially satisfies \ref{hyp:Nexist}, since the origin is a fixed point for the dynamics. A counterexample to \ref{hyp:Nexist_star}, if it exists, looks much harder to obtain.
\end{remark}

\begin{remark} \label{rem:a5idea}
	The nature of conditions \ref{hyp:Nexist} and \ref{hyp:Nexist_star} becomes clearer if we consider how rotational properties are normally used to recover periodic solutions.
	A key step in classical approaches is to consider rotation, defined as in \eqref{eq:rotdef},  for all the solutions starting from a closed curve surrounding the origin (i.e.~the point around which rotation is defined). This is well illustrated, for instance, by Theorem \ref{th:deg1}, where such a curve is provided by (a suitable part of) $\partial U$. The Poincaré--Birkhoff Theorem is another example, considering two of such curves. 
 This however implies that the solutions passing through the origin $0_{\R^2}$ during $[0,T]$ are \lq\lq trapped inside\rq\rq\ the orbit of such closed curve, therefore each of them exists on the whole interval, and thus \ref{hyp:Nexist} holds. On the other hand, if the origin is not \lq\lq trapped\rq\rq\ by $\partial U$, the topological structure changes drastically and, in general, the fixed-point may be lost. For instance, if in Theorem \ref{th:deg1} we consider instead the case $0_{\R^2}\notin \clos U$, we obtain $\deg(\phi(T,\cdot)-\I_{\R^2},U,0)=0$, which is not useful to our aims.
\end{remark}

\begin{remark}\label{rem:a5conj}
	Corollary \ref{cor:chall_star} leaves open the question whether \ref{hyp:Nexist_star} is always satisfied, or under which conditions on $f$ and $p$. We observe, however, that possible counterexamples would be pathological, with noncontinuability of solutions widespread everywhere on the phase plane.
	Several results in literature illustrate encouraging continuability properties for  \eqref{eq:introODE}, although none is sufficient for \ref{hyp:Nexist}.  In addition to the discussion at the end of the introduction on bounded (and therefore global) solutions, we mention for instance \cite{Elias}, where the existence of a global solution for \eqref{eq:introODE} is obtained under very general assumptions.
	
	Finally, we observe that the situation may be more favourable in the  Hamiltonian case $p\equiv 0$, due to the additionally available variational structure, which could possibly be helpful to verify or replace \ref{hyp:Nexist_star}. Indeed, in such a case the existence of multiple periodic solutions can be obtained also by variational methods, cf.~\cite{BaBe,Rab,Long} which studied the case of an autonomous $f=f(x)$. Noncontinuability still poses a severe issue also for variational methods, yet the structural role of \ref{hyp:Nexist_star} discussed in Remark~\ref{rem:a5idea} might be less pivotal following an alternative approach.
\end{remark}

\section{Proof of Theorem \ref{th:chall}}
We begin by recalling that \eqref{eq:Fode} is equivalent to a planar first order system: more precisely, we introduce the associated Cauchy problem
\begin{equation} \label{eq:Fcauchy}
	\begin{cases}
\dot x =y\\
\dot y = G(t,x,y):= -f(t,x)-p(t,x,y)\\
x(t_0)=x_0,\quad y(t_0)=y_0
	\end{cases}
\end{equation}
By \ref{hyp:reg} and \ref{hyp:uniqueness}  there is local existence and uniqueness of solution for the problem \eqref{eq:Fcauchy} for every initial value $(x_0,y_0)\in \R^2$. On the other hand, global existence of a solution, hence in particular its existence on $[0,T]$, may fail, see \cite{CofUll} for a counterexample.

We define the set $\Delta\subset[0,T]$ as
\begin{equation}
	\Delta:=\{(t,t_0)\subset[0,T]^2\quad\text{such that $t\geq t_0$}\}\,.
\end{equation}
We then introduce the evolution map  $\zeta\colon\Delta\times\R^2\to\R^2\cup\{\infty\}$ of \eqref{eq:Fcauchy}, defined as follows
\begin{equation}
	\zeta(t;t_0,x_0,y_0)=\begin{cases}(x(t),y(t)) &\text{if the solution $(x(s),y(s))$ of \eqref{eq:Fcauchy} exists on $[t_0,t]$; }\\
		\infty &\text{otherwise.}
		\end{cases}
\end{equation}
As a consequence of Theorem \ref{th:contdep}, the function $\zeta\colon \Delta\times\R^2\to\R^2\cup\{\infty\}$ is continuous, where, as usual, a base of neighbourhoods for $\infty$ is given by the sets $\R^2\cup\{\infty\}\setminus B_r$, for $r>0$.

The case $t_0=0$ is the most relevant on our analysis, so we write
\begin{equation}\label{eq:Fphidef}
	\phi(t,x_0,y_0):=\zeta(t;0,x_0,y_0)\colon [0,T]\times\R^2\to\R^2\cup\{\infty\}\,.
\end{equation}
Our plan is to apply  Theorem \ref{th:deg1}  to $\phi(T,\cdot)$, since fixed-points of $\phi(T,\cdot)$ correspond to $T$-periodic solutions of \eqref{eq:Fcauchy}.

Our approach is based on estimates of the rotation $\rho$ around the origin performed by a solution in given time interval. Such notion can be defined only if the solution does not cross the origin during such interval. To identify such situations, we define, for $t\in[0,T]$, the family of sets $\Null_t\subset\R^2$ as
\begin{equation}
	\Null_t=\{(x,y)\in\R^2\,\text{such that $\zeta(s;t,x,y)=0_{\R^2}$ for some $s\in[t,T]$}\}\,.
\end{equation}
Notice that $0_{\R^2}\in\Null_t$ and that $\Null_t$ is a compact set by \ref{hyp:Nexist} and Theorem~\ref{th:contdep}. We also write $\Null=\Null_0$, in accordance with \eqref{eq:defNull}.
Moreover, as a consequence of Theorem \ref{th:contdep}, using also \ref{hyp:Nexist} and the compactness of $\Null$, we observe that
\begin{enumerate}[label=\textup{(\roman*)}]
	\item there exists an open set $U_0$ such that $\Null\subset U_0$ and for every $(x_0,y_0)\in U_0$ the solution of \eqref{eq:Fcauchy} is defined on the whole interval $[0,T]$. \label{prop:Nneigh}
\end{enumerate} 
Setting
\begin{align*}
	\Omega=\{(s;t,\theta,r)\in\Delta\times \Hs\text{ such that $\zeta(s;t,\Psi(\theta,r))\neq\infty$ and $\Psi(\theta,r)\notin \Null_t$}\} ,
\end{align*}
by the homotopy lifting property and Theorem \ref{th:contdep} we deduce that there exists a unique continuous map $Z\colon\Omega\to\Hs$ such that
$Z(t;t,\theta,r)=(\theta,r)$ and
\begin{equation*}
	\Psi(Z(s;t,\theta,r))=\zeta(s;t,\Psi(\theta,r)) \qquad\text{for every $(s,t,\theta,r)\in\Omega$\,.}
\end{equation*}
Let us identify the two components of $Z$ as $Z=(Z^\theta,Z^r)$.
We are now ready to introduce a generalized notion of rotation for the solutions of \eqref{eq:Fcauchy}. Writing
\begin{equation*}
	\Gamma=\{(s,t,x,y)\quad\text{such that $(s,t)\in\Delta$ and $(x,y)\in(\R^2\setminus\Null_t)$}\}
\end{equation*}
 we define the function $\rho(s;t,x,y)\colon \Gamma\to (-\infty,+\infty]$ as 
\begin{equation}
\rho(s;t,\Psi(\theta,r))=\begin{cases}
	\displaystyle\frac{Z^\theta(s;t,\theta, r)-\theta}{2\pi}&\text{if $(s,t,\theta,r)\in\Omega$\,;}\\[1mm]
	+\infty &\text{otherwise.}
\end{cases}
\end{equation}
Notice that $\rho$ is well-defined, since, fixed any $(x,y)\in\R^2\setminus\Null_t$ with $\zeta(s;t,x,y)\neq\infty$, the value of  $Z^\theta(s;t,\theta,r)-\theta$ does not depend on the choice of $(\theta,r)\in\Psi^{-1}(x,y)$. We also remark that, with the notation of Section \ref{sec:prel}, if $\rho(T;0,x,y)<+\infty$ then $\rho(T;0,x,y)=\rot_\phi((x,y))$.

Our plan now is to prove that $\rho$ is continuous and that $\rho(T;0,x,y) \to+\infty$ as $\abs{(x,y)}\to +\infty$. We begin with some preliminary properties.

\begin{prop} The following statements are true:
\begin{enumerate}[label=\textup{(\roman*)},resume]
	\item if $\rho(t;t_0,x_0,y_0)=a$, then $\rho(s;t_0,x_0,y_0)\geq a-\frac{1}{2} $ for every $s>t$; \label{prop:monotrot}
	\item if $\rho(\bar t;t_0,x_0,y_0)=+\infty$, let us define \begin{equation*}
		t_\mathrm{max}:=\sup\{t\in[t_0,T]:\rho(t;t_0,x_0,y_0)<+\infty\}\,.
	\end{equation*} Then $\displaystyle
	\lim_{t\to t_\mathrm{max}} \rho(t;t_0,x_0,y_0) =+\infty$.
	\label{prop:rotinf}
\end{enumerate}
\end{prop}
Notice that $t_\mathrm{max}>0$ in \ref{prop:rotinf} is well defined, since $\rho(t_0;t_0,x_0,y_0)=0$. Also, let us  denote for simplicity the orbit of the solution associated to the initial values $(t_0,x_0,y_0)$ as $\zeta(\cdot;t_0,x_0,y_0)=:\zeta^0(\cdot)=(\zeta^0_x,\zeta^0_y)(\cdot)$, where the last decomposition applies only when $\zeta^0$ has finite value.
\begin{proof}

We prove first \ref{prop:monotrot}.  In order to have $\rho(s;t_0,x_0,y_0)< a-\frac{1}{2}$ for some $s>t$, then the orbit $\zeta^0(\cdot)$ would have to cross the $y$-axis counter-clockwise This is not possible, since from the first equation in \eqref{eq:Fcauchy} it is clear that the $y$-axis can be crossed only clockwise. Observe that since $(x_0,y_0)\notin \Null_{t_0}$  the orbit $\zeta^0(\cdot)$ cannot cross the origin $0_{\R^2}$. Moreover, by construction, if $\rho(s;t_0,x_0,y_0)=+\infty$ then $\rho(\tilde s;t_0,x_0,y_0)=+\infty$ for every $\tilde s\in[s,T]$. Hence, the statement \ref{prop:monotrot} is true.

Suppose now by contradiction that \ref{prop:rotinf} is false. Then, since the $y$-axis can be crossed only clockwise outside the origin, it has to be crossed only a finite number of times by the orbit $\zeta^0(\cdot)$. Let us therefore take $\tilde t\in[t_0,t_\mathrm{max})$ such that $\zeta^0_x(t)\neq 0$ for every $t\geq \tilde t$. We discuss only the case $\zeta^0_x(\tilde t)>0$, since the case $\zeta^0_x(\tilde t)<0$ is symmetric.

Let us notice that by \ref{hyp:reg},\ref{hyp:growth} and \ref{hyp:plim} there exists a positive  function $\gamma_1\in L^1([0,T],\R)$ such that
\begin{equation*}
	G(t,x,y)<\gamma_1(t) \quad\text{for every $(t,x,y)\in[0,T]\times[0,+\infty)\times\R$\,.}
\end{equation*} 
Precisely, we use \ref{hyp:growth} to obtain the existence of $x^*>0$ such that $-f(t,x)+C^px<0$ for every $x>x^*$ and $t\in[0,T]$. We conclude setting $\gamma_1=\gamma_p+m_D$, where $m_D$ is given by the third condition of the Carathéodory property for $f$ with respect to the compact set $D=[0,T]\times[0,x^*]$.
It follows that \begin{equation*}
	\zeta^0_y(t)<\zeta^0_y(\tilde t)+\int^{t_\mathrm{max}}_{\tilde t}\gamma_1(\tau)\dd \tau:=y_\mathrm{max} \qquad \text{for every $t\in[\tilde t,t_\mathrm{max})$\,.}
\end{equation*} 
This, consequently, implies that $\zeta_x^0(t)<\zeta^0_x(\tilde t)+y_\mathrm{max}(t_\mathrm{max}-\tilde t)=:x_\mathrm{max}$  for every $t\in[\tilde t,t_\mathrm{max})$. 
Reasoning as above, by \ref{hyp:reg},\ref{hyp:growth} and \ref{hyp:plim} there exists a positive function $\gamma_2\in L^1([0,T],\R)$ such that
\begin{equation*}
	G(t,x,y)>-\gamma_2(t) \quad\text{for every $(t,x,y)\in[0,T]\times[0,x_\mathrm{max}]\times\R$\,.} 
\end{equation*} 
Hence, $\zeta^0_y(t)>\zeta^0_y(\tilde t)-\int^{t_\mathrm{max}}_{\tilde t}\gamma_2(\tau)\dd \tau=:y_\mathrm{min}$  for every $t\in[\tilde t,t_\mathrm{max})$. 

Since we have shown that $\zeta^0(t)$ remains within the compact set $[0,x_\mathrm{max}]\times[y_\mathrm{min},y_\mathrm{max}]$ for every $t\in[\tilde t,t_\mathrm{max})$, we deduce that the solution of \eqref{eq:Fcauchy} has a continuation to an interval $[0,t_\mathrm{max}+\epsilon)$ for some $\epsilon>0$ (cf.~for instance \cite[Ch.~I, Th.~5.2]{Hale}), reaching the desired contradiction.
\end{proof}

\begin{lemma} \label{lemma:myrotcont}
	The function $\rho$ is continuous.
\end{lemma}
\begin{proof}
First of all, let us notice that $\rho$ is continuous at each point $(s,t,x,y)$ such that $\zeta(s;t,x,y)\neq\infty$. This follows from Theorem \ref{th:contdep} and the fact that the restriction of $\Psi$ to $\Hs$ is a local homeomorphism.

It remains to show that $\rho$ is continuous where its value is $+\infty$. Let us consider $(t,t_0,x_0,y_0)\in\Gamma$ such that $\rho(t;t_0,x_0,y_0)=+\infty$. We now show that for every $j\in\N$ there exists a neighbourhood $[t_j,T]\times V_0\subset \Gamma$, with $t_j<t\leq T$,  such that $\rho(s;t,x,y)>j$ for every $(s,t,x,y)\in [t_j,T]\times V_0$.

By \ref{prop:rotinf}, there exist $t_{j}<t$ such that $\rho(t_j;t_0,x_0,y_0)=a>j+1$ with $a\in\R$. Then, by the continuity of $\rho$ in $(t_j,t_0,x_0,y_0)$, which we have by the first part of the proof, we know that there exists a neighbourhood $V_0$ of $(t_0,x_0,y_0)$ such that $\rho(t_j,t,x,y)>j+1$ for every $(t,x,y)\in V_0$. By \ref{prop:monotrot}, we deduce that $\rho(s,t,x,y)>j+\frac{1}{2}$ for every $(s,t,x,y)\in [t_j,T]\times V_0$, completing the proof.
\end{proof}

\begin{lemma} \label{lemma:stimarot}
	For every $n\in\N$ there exists a radius $\widetilde{R}_n>0$ such that, if $(x,y)\in \R^2\setminus \Null$ and $\widetilde{R}_n<\abs{\phi(t,x,y)}\neq\infty$ for every $t\in[0,T]$, then  $\rho(T;0,x,y)>n$.
\end{lemma}
\begin{proof}
	Let us take $(x,y)\in \R^2\setminus \Null$ such that $\widetilde{R}_n<\phi(t,x,y)\neq\infty$ for every $t\in[0,T]$ and a value $\widetilde{R}_n$ yet to be assigned.
	As a consequence of \ref{hyp:reg},\ref{hyp:growth} and \ref{hyp:plim}, for every $\alpha>1$ there exist a positive function $\gamma_\alpha\in L^1([0,T],\R)$ such that 
	\begin{equation}\label{eq:super_est}
		(f(t,x)+p(t,x,y))x>\alpha x^2-\gamma_\alpha(t)\abs{x} \quad\text{for every $(t,x,y)\in[0,T]\times \R^2$}
	\end{equation}
We set $b_\alpha=\frac{1}{T}\int_0^T\gamma_\alpha(\tau)\dd \tau\geq0$.

Let us fix any $(\theta, r)\in\Psi^{-1}(x,y)$. For the computation in this proof, we write $\tilde \theta (t):=Z^\theta(t;0,\theta,r)$ and $(\tilde x(t),\tilde y(t)):=\zeta(t;0,x,y)$. Since we are using the polar coordinates \eqref{eq:Fpolar}, which are a local diffeomorphism outside the origin, \eqref{eq:super_est} yields
\begin{align*}
	\dot{\tilde \theta}(t)&=\frac{\bigl[f(t,\tilde x(t))+p(t,\tilde x(t),\tilde y(t))\bigr]\tilde x(t)+\tilde y^2(t)}{\tilde x^2(t)+\tilde y^2(t)}\\
	&>
	\frac{\alpha \tilde x^2(t)+\tilde y^2(t)}{\tilde x^2(t)+\tilde y^2(t)}-\frac{\gamma_\alpha(t)\abs{\tilde x(t)}}{\tilde x^2(t)+\tilde y^2(t)}\geq \alpha\cos^2\tilde\theta(t)+\sin^2\tilde\theta(t)-\frac{\gamma_\alpha(t)}{\widetilde{R}_n}\,.
\end{align*}

Let us take $k\in\Z$ such that $2\pi(k+1)\geq \tilde\theta(T)-\tilde\theta(0)> 2\pi k$; by definition we have 
\begin{equation}
	\rho(T;0,x,y)=\frac{\tilde\theta(T)-\tilde\theta(0)}{2\pi}> k \,.
\end{equation}
Writing for $\alpha>1$
\begin{equation}
	L_\alpha:=\int_0^{2\pi}\frac{\dd\beta}{\alpha \cos^2 \beta+\sin^2 \beta}=\frac{2\pi}{\sqrt{\alpha}}\,,
\end{equation}
we have
\begin{equation*}
2\pi (k+1)L_\alpha\geq\int_{\tilde\theta(0)}^{\tilde\theta(T)}\frac{\dd\beta}{\alpha \cos^2 \beta+\sin^2 \beta}=\int_{0}^{T}\frac{\dot{\tilde \theta}(t)\dd t}{\alpha \cos^2 \tilde\theta(t)+\sin^2 \tilde\theta(t)}>T\left(1-\frac{b_\alpha}{\widetilde{R}_n}\right)
\end{equation*}
where in the estimate of the last term we have used that $\alpha \cos^2 \beta+\sin^2 \beta\geq 1$ for every $\beta\in\R$, since $\alpha>1$. Thus
\begin{equation}
k>\frac{T}{2\pi L_\alpha}\left(1-\frac{b_\alpha}{\widetilde{R}_n}\right)-1=\frac{T\sqrt{\alpha}}{4\pi^2 }\left(1-\frac{b_\alpha}{\widetilde{R}_n}\right)-1\,.
\end{equation}
Fixing now $\alpha>1$ and $\widetilde{R}_n>0$ such that
\begin{align}\label{eq:rot_est_req}
	\sqrt{\alpha}>\frac{8\pi^2n}{T} && \widetilde{R}_n>2b_\alpha
\end{align}
we obtain that $k>n-1$; since $k$ and $n$ are both integers, it follows that $k\geq n$ and therefore $\rho(T;0,x,y)>n$. Notice also that the requirements \eqref{eq:rot_est_req} depend on $n,g,p$, but not on the specific choice of $(x,y)$.
\end{proof}

\begin{lemma}\label{lemma:rotelast}
For every $n\in\N$ there exists a radius $R_n>0$ such that, for every $(x,y)\in \R^2$ with $\abs{(x,y)}>R_n$, we have $\rho(T;0,x,y)>n$.
\end{lemma}
\begin{proof}
	Let us set 
\begin{equation*}
R_\Null=1+\max\left\{\abs{\phi(t,x,y)}\;\text{such that $(t,x,y)\in[0,T]\times\Null$}\right\}<+\infty\,.
\end{equation*}
Notice that $R_\Null$ is well defined thanks to \ref{hyp:Nexist}. We take $\widetilde{R}_n$ from Lemma \ref{lemma:stimarot} and, without loss of generality, assume $\widetilde{R}_n>R_\Null$.
We make the following claim: \medskip

 \emph{ There exists $R_n>\widetilde R_n$	such that each point $(x,y)\in \R^2$ with $\abs{(x,y)}>R_n$  satisfies one of the two following alternative properties:
\begin{itemize}
	\item $\abs{\phi(t,x,y)}>\widetilde{R}_n$ for every $t\in[0,T]$ (including the case $\phi(t,x,y)=\infty$)
	\item there exists $(\bar t,\bar x, \bar y)\in[0,T]\times\R^2$ such that $\phi(\bar t,x,y)=(\bar x, \bar y)$, $\abs{(\bar x,\bar y)}<\widetilde{R}_n$ and
	\begin{equation} \label{eq:rot_interelastic}
		\rho(\bar t;0,x,y)>n+1\,.
	\end{equation}
\end{itemize}}\noindent
Please notice that $\rho(\cdot;\cdot,x,y)$ is well defined since $(x,y)\notin\Null$ due to $R_n>\widetilde R_n>R_\Null$.

If the claim is true, then Lemma \ref{lemma:rotelast} follows straightforwardly. Indeed, in the first case either $\phi(T,x,y)=\infty$, hence $\rho(T;0,x,y)=+\infty$, or $\phi(T,x,y)\neq \infty$, thus $\rho(T;0,x,y)>n$ follows from Lemma \ref{lemma:stimarot}. If instead the second option holds, then $\rho(T;0,x,y)>n$ follows from \eqref{eq:rot_interelastic} and \ref{prop:monotrot}.

Thus, it remains to show that the claim is true. Let us take any $(x,y)$ with $\abs{(x,y)}>R_n$. We observe that either the first option of the claim holds, or there exists $(\bar t,\bar x, \bar y)\in[0,T]\times\R^2$ such that $\phi(\bar t,x,y)=(\bar x, \bar y)$ and $\abs{(\bar x,\bar y)}<\widetilde{R}_n$. Hence, we just have to prove that, in the latter case, for a suitable choice of $R_n$ the estimate \eqref{eq:rot_interelastic} holds.
	
Let us consider the Cauchy problem
\begin{equation}\label{eq:FCau_elastic}
	\begin{cases}
		\displaystyle\frac{\dd \hat x}{\dd \hat t}  =\hat y\\[4mm]
		\displaystyle\frac{\dd \hat y}{\dd \hat t} = -G(T-\hat t,-\hat x,\hat y)=:\widehat G(\hat t,\hat x,\hat y)\\[4mm]
		\hat x(\hat t_0)=\hat x_0,\quad \hat y(\hat t_0)=\hat y_0
	\end{cases}
\end{equation}
Firstly, we notice that the solution of \eqref{eq:FCau_elastic} for any choice of the Cauchy condition corresponds to the solution of \eqref{eq:Fcauchy} through the change of variables
\begin{align*}
	\hat x&=-x &\hat y&=y & \hat t&= T-t\\
	\hat x_0&=-x_0 &\hat y_0&=y_0 & \hat t_0&= T-t_0
\end{align*}
Secondly, we observe that, writing accordingly $\hat f(\hat t,\hat x):=-f(T-\hat t,-\hat x)$ and $\hat p(\hat t,\hat x,\hat y):=-p(T-\hat t,-\hat x,\hat y)$, the assumptions \ref{hyp:reg},\ref{hyp:uniqueness},\ref{hyp:growth},\ref{hyp:plim} and \ref{hyp:Nexist} are satisfied also by the new dynamics \eqref{eq:FCau_elastic}. Hence, we can define analogously $\widehat\zeta$, $\widehat{\Null_{\hat t}}$ and $\widehat\rho$ for system \eqref{eq:FCau_elastic} and apply the related results proved above.

In particular, we notice that 
\begin{equation}\label{eq:rot_equiv}
\rho(t;0,x,y)=\widehat\rho(T;T-t,-x^*,y^*) \qquad\text{where $(x^*,y^*)=\phi(t,x,y)$\,.}
\end{equation}
As a consequence of Theorem \ref{th:contdep}, \ref{hyp:Nexist} and Lemma \ref{lemma:myrotcont} applied to the dynamics \eqref{eq:FCau_elastic}, for every $(\hat t,\hat x,\hat y)\in[0,T]\times \clos{B_{\widetilde{R}_n}}$ we are in one of the following cases
	\begin{enumerate}[label=\textup{\alph*)}]
 		\item $(\hat x,\hat y)\in\widehat{\Null_{\hat t}}$ or  $\widehat{\rho}(T;\hat t,\hat x,\hat y)\leq n+1$. In both situations there exists a neighbourhood $V_{\hat t,\hat x,\hat y}$ of $(\hat t,\hat x,\hat y)$ and a constant $C_{\hat t,\hat x,\hat y}$ such that $\abs{\widehat\zeta(T;s,u,w)}<C_{\hat t,\hat x,\hat y}$ for every $(s,u,w)\in V_{\hat t,\hat x,\hat y}$\,;
		\item $\widehat{\rho}(T;\hat t,\hat x,\hat y)> n+1$. Hence, there exists a neighbourhood $V_{\hat t,\hat x,\hat y}$ of $(\hat t,\hat x,\hat y)$ such that $\widehat{\rho}(T;s,u,w)> n+1$ for every $(s,u,w)\in V_{\hat t,\hat x,\hat y}$. In this case we set the constant $C_{\hat t,\hat x,\hat y}=1$\,. \label{prop:rot_elast_second}
	\end{enumerate}
By the compactness of $[0,T]\times \clos{B_{\widetilde{R}_n}}$, there exists a finite collection $V_{\hat t_k,\hat x_k,\hat y_k}$, with $k=0,\dots,K$ which is a finite cover of $[0,T]\times \clos{B_{\widetilde{R}_n}}$.

Let us set
\begin{equation}
	R_n:=\max\{\widetilde{R}_n+1,C_{\hat t_0,\hat x_0,\hat y_0},C_{\hat t_1,\hat x_1,\hat y_1},\dots,C_{\hat t_K,\hat x_K,\hat y_K}\}\,.
\end{equation}
We now consider $(\bar t, \bar x, \bar y)$ as in the second point of the claim and show that for this choice of $R_n$ the estimate \eqref{eq:rot_interelastic} holds. Indeed, by \eqref{eq:rot_equiv} we have
\begin{equation}\label{eq:rot_equiv_bis}
	\rho(\bar t;0,x,y)=\widehat{\rho}(T;T-\bar t,-\bar x,\bar y)\,.
\end{equation}
Since $\abs{(-\bar x,\bar y)}<\widetilde{R}_n$ and
\begin{equation*}
	\abs{\widehat\zeta(T;T-\bar t,-\bar x,\bar y)}=\abs{(x,y)}>R_n\,,
\end{equation*}
we deduce that the triplet $(T-\bar t,-\bar x,\bar y)$ is included in case \ref{prop:rot_elast_second} above, hence $\widetilde{\rho}(T;T-\bar t,-\bar x,\bar y)>n+1$ and by \eqref{eq:rot_equiv_bis} we deduce \eqref{eq:rot_interelastic}, concluding the proof.
\end{proof}

Notice that, if we had global existence of solution, Lemma \ref{lemma:rotelast} would follow more easily from Lemma \ref{lemma:stimarot}. More precisely, in the claim within the proof of Lemma~\ref{lemma:rotelast} we could consider only the first option: such simplified version of the claim is sometimes referred as \emph{elastic property} of the dynamics.

We are now ready to complete the proof of the main theorem.

\begin{proof}[Proof of Theorem \ref{th:chall}]
	First of all, notice that the thesis is equivalent to the existence of a $T$-periodic solution for the planar system introduced in the Cauchy problem \eqref{eq:Fcauchy}. For such planar dynamics, we consider the function $\phi$ as in \eqref{eq:Fphidef}; our aim is to find a fixed point of $\phi(T,\cdot)$. 
	
	By \ref{prop:Nneigh} there exists a bounded open set $U_1\subset \R^2$ such that
	\begin{equation}
		0\in\Null\subset U_1 \subset \clos{U_1} \subset U_0\,.
	\end{equation}
	In particular, we notice that $\phi(T,z)\neq\infty$ for every $z\in \partial U_1$, hence $\rot_\phi(z)=\rho(T;0,z)$ is well-defined and finite for every $z=(x,y)\in\partial U_1$.
	By Lemma \ref{lemma:myrotcont} and the compactness of $\partial U_1$, we deduce that there exists $\bar n\in \N$ such that 
	\begin{equation} \label{eq:Frotout}
			\rot_\phi(\partial U_1)<\bar n\,.
	\end{equation}
	We now define
	\begin{equation}
		U=U_1 \cup \left\{z\in \R^2\setminus\Null \,\text{such that $\rho(T;0,z)<\bar n+\displaystyle\frac{1}{2}$}\right\}\,.
	\end{equation}
	By Lemma \ref{lemma:myrotcont} we deduce that $U$ is open and by Lemma \ref{lemma:rotelast} that it is bounded. 
	 Moreover, by Lemma \ref{lemma:myrotcont} and since $\Null\cap \partial U=\emptyset$, we have
	\begin{equation} \label{eq:Frotinn}
		\rot_\phi(\partial U)=\left\{\bar n +\frac{1}{2}\right\}, \qquad\text{hence}\quad \rot_\phi(\partial U)\cap \Z=\emptyset\,.
	\end{equation}
By Theorem \ref{th:deg1} we deduce that $\deg(f_T,U,0)=1$ and, in particular, that $\phi(T,\cdot)$ has a fixed point in $U$, which corresponds to a $T$-periodic solution of~\eqref{eq:Fode}.
\end{proof}

\textbf{Acknowledgements.}
The Author is grateful to A.~Fonda and A.~Boscaggin for the valuable discussions on the topic. The Author is partially supported by the GA\v{C}R Junior Star Grant 21-09732M.


\end{document}